\documentclass[onefignum,onetabnum]{siamart220329}

\usepackage{amsmath,amssymb,amsfonts}
\usepackage{graphicx}
\usepackage{booktabs}
\usepackage{url}
\usepackage{xcolor}

\makeatletter
\@ifundefined{cref@constructprefix}{}{%
  \def\refstepcounter@optarg[#1]#2{%
    \cref@old@refstepcounter{#2}%
    \cref@constructprefix{#2}{\cref@result}%
    \@ifundefined{cref@#1@alias}%
      {\def\@tempa{#1}}%
      {\def\@tempa{\csname cref@#1@alias\endcsname}}%
    \protected@edef\cref@currentlabel{%
      [\@tempa][\arabic{#2}][\cref@result]%
      \csname p@#2\endcsname\csname the#2\endcsname}%
  }%
}
\makeatother

\newif\ifmarked
\markedfalse
\ifmarked
  \newcommand{\rev}[1]{{\color{blue}#1}}
  \newenvironment{revblock}{\color{blue}}{}
\else
  \newcommand{\rev}[1]{#1}
  \newenvironment{revblock}{}{}
\fi

\newcommand{\C}{\mathbb{C}}
\newcommand{\eps}{\varepsilon}
\newcommand{\dA}{\,dA}

\graphicspath{{figures/}{./}}

\begin{document}

\title{Cubature from rational approximation}

\author{Gentian Zavalani\thanks{Technische Universit{\"a}t Dresden, %
  Institute of Numerical Mathematics, %
  01062 Dresden, Germany
(\email{gentian.zavalani@tu-dresden.de}).}}

\headers{Cubature from rational approximation}{G.\ Zavalani}

\maketitle

\begin{abstract}
We present a numerical construction of cubature rules for area
integrals of analytic functions over planar domains with rectifiable
Jordan boundary. The starting point is the Cauchy--Green identity.
Given a weight $w$, we choose a $\bar\partial$-antiderivative $W$ and
reduce the area integral to a contour integral involving the boundary
values of $W$. These values are then approximated by a rational
function with free poles, computed by the AAA algorithm. The poles
inside the domain become cubature nodes, the corresponding residues
become weights, and the boundary residual controls the error through
an a posteriori estimate, rigorous once the continuous boundary
residual is bounded. \rev{Under an exterior-zero condition, the same rule}
admits a dual reading, as the exact
integral of a rational interpolant to the integrand, the area analogue
of the one-dimensional interpolatory viewpoint.
The numerical examples recover the disk mean-value rule and the
focal-segment rule of the ellipse to machine precision, reproduce the
exact finite quadrature identities of quadrature domains with both
separated and confluent nodes, and evaluate logarithmic and Cauchy
volume potentials. The interior poles trace
analytic skeletons that we identify tentatively with the mother bodies
of potential theory, along with reflected image points; for the square the observed convergence is root-exponential.
\end{abstract}



\section{Introduction}
\label{sec:intro}

In a recent paper, Horning and Trefethen showed how a great variety of
quadrature formulas in one dimension can be generated from rational
approximations of the Cauchy transform of a weight function
\cite{HorningTrefethen2026}. The poles of the rational function become
the quadrature nodes and the residues become the weights, and the whole
construction lives on arcs and contours in the complex plane. This
paper asks the corresponding two-dimensional question. Can quadrature
rules for \emph{area} integrals,
\begin{equation}
\label{eq:areaint}
I \;=\; \iint_\Omega f(z)\, w(z, \bar z) \dA ,
\end{equation}
over a domain $\Omega \subset \C$ with rectifiable Jordan boundary
$\gamma = \partial\Omega$, be generated in the same spirit? Throughout,
we assume that $f$ is analytic in a neighbourhood of $\bar\Omega$.

Remarkably, a classical identity already supplies the transform: it
converts \eqref{eq:areaint} into a problem of approximation on the
boundary. Suppose $W(z,\bar z)$ satisfies
$\partial W/\partial\bar z=w$. Then the Cauchy--Green identity, which
underlies the Cauchy--Pompeiu formula \cite{Pompeiu}, gives, for $f$
analytic in a neighbourhood of $\bar\Omega$,
\begin{equation}
\label{eq:pompeiu}
\iint_\Omega f(z)\, w(z, \bar z) \dA
\;=\; \frac{1}{2i} \oint_\gamma f(z)\, W(z, \bar z)\, dz
\end{equation}
\rev{(the regularity needed of $W$ is spelled out in
section~\ref{sec:construction})}.
The restriction of $W$ to $\gamma$ is known boundary data. In the
simplest case of a pure area integral we have $w \equiv 1$ and
$W = \bar z$, so the data is simply the complex conjugate of the
boundary parametrization. Now let $\rho$ be a rational function with
simple poles $z_k$ that approximates this data on $\gamma$. The poles
inside $\Omega$ contribute residues, while the poles outside $\Omega$,
together with any polynomial part, integrate to zero against analytic
$f$. Residue calculus thus turns \eqref{eq:pompeiu} into a
\emph{cubature rule},
\begin{equation}
\label{eq:rule}
I_n \;=\; \pi \sum_{z_k \in \Omega} c_k\, f(z_k),
\end{equation}
whose nodes are the interior poles and whose error is controlled
entirely by the quality of the fit on the boundary
(Theorem~\ref{thm:cert}). We compute the poles with the AAA algorithm
\cite{NST2018,NakatsukasaTrefethen2025,TrefethenWilber2025} and the
coefficients by a linear least-squares fit on the boundary samples. We
call \eqref{eq:rule} the \emph{AAA cubature rule}.

As in the one-dimensional case, the nodes and weights depend only on
$W$. Once $W|_\gamma$ has been approximated,
the same rule applies to every analytic $f$ for which
\eqref{eq:pompeiu} holds. Any $W+h$ with $h$ analytic is an antiderivative as well, and it serves
equally, since
$\oint_\gamma f(z)h(z)\,dz=0$ leaves the exact contour integral
unchanged. At a finite rational degree, however, different choices of
$W$ lead to different numerical fits and hence to different computed
rules, and in each example below we simply use \rev{a} natural choice,
stated explicitly. \rev{For a real weight $w$ the rule also handles
harmonic integrands $u=f+\bar g$ with $f$ and $g$ analytic, because
$\iint_\Omega \bar g\, w\dA=\overline{\iint_\Omega g\, w\dA}$: the
approximation is $\pi\sum_k \bigl(c_k f(z_k)+\bar c_k\,
\overline{g(z_k)}\bigr)$, with the same nodes and conjugated weights,
and on the ellipse of section~\ref{sec:classical} it is accurate to
$3\times10^{-16}$ for $u=e^z+\overline{\cos z}$. The construction is thus
one for analytic integrands and, through conjugation, harmonic ones,}
and the weights are in general complex.

\rev{By the Cauchy--Green identity, the whole problem moves to the
boundary, so that the quantity
entering the error bound of Theorem~\ref{thm:cert} is the boundary
residual $\|\rho-W\|_{\gamma,\infty}$ alone. The sampled validation
residual used in the computations is an a posteriori indicator of this
quantity; it becomes a certificate once the continuous boundary supremum
has been bounded.}

The location of the nodes \rev{can reflect} the analytic continuation of the boundary data into $\Omega$. \rev{When $w\equiv1$ and $\gamma$ is real-analytic, by the
Cauchy--Kovalevskaya theorem~\cite{Evans2010} the data $\bar z$
continues analytically from $\gamma$ into a neighbourhood of $\gamma$ as
the Schwarz function $S(z)$ of the curve \cite{Davis1974,Shapiro1992}. Continued further into $\Omega$,
$S$ may meet branch points, poles, or a natural boundary, and which
singularities occur, and where, is known only for special curves.} \rev{For the ellipse considered here, the foci are branch points of $S$, and the focal
segment is a branch cut joining them. The computed interior poles line up along this segment,} tracing out what potential theorists call a mother body,
an analytic skeleton of the domain
\cite{Zidarov,Gustafsson1998,GustafssonPutinar}. Domains whose Schwarz
function is actually meromorphic in $\Omega$ are the \emph{quadrature
domains} of Aharonov and Shapiro
\cite{AharonovShapiro1976,Sakai1982,GustafssonShapiro2005}; these are
the domains that satisfy exact finite quadrature identities. The construction is thus a direct numerical route from boundary
data to approximate quadrature identities. The inverse problem, that of
reconstructing a domain from its moments, has a substantial literature
\cite{GHMP2000,GPSS2009}, and one may regard what we do here as its
forward counterpart.

The work closest to ours is Trefethen's numerical computation of the
Schwarz function \cite{Trefethen2025Schwarz}, which applies AAA to
$\bar z$ on a curve and reads off branch cuts from the strings of poles
that appear. Horning and Trefethen \cite{HorningTrefethen2026}, as
already noted, obtain one-dimensional quadrature from rational
approximation of a Cauchy transform. The present paper joins these two
ideas through the Cauchy--Green identity: rational approximation of a
$\bar\partial$-antiderivative on the boundary produces a weighted
cubature rule whose error is controlled by the boundary residual.

Boundary reduction has a long history. It underlies Gauss--Green
cubature \cite{SommarivaVianello2009}, which applies product Gauss
rules along the boundary, and connections between approximation and
analytic skeletons appear both in the zeros of Bergman polynomials
\cite{GPSS2009} and in explicit constructions of mother bodies for
special geometries \cite{Gustafsson1998,SavinaSterninShatalov2005}.
Rational approximation adds a direct numerical construction that attaches weights to the interior poles.

Throughout, we test the construction against exact identities,
independent reductions, adaptive area integration, and two-resolution
reference computations. Section~\ref{sec:classical} treats the
classical examples: it recovers the disk mean-value rule and the
focal-segment reduction of an ellipse, and it exhibits a candidate
five-armed skeleton for a starfish domain. Section~\ref{sec:qd} turns
to quadrature domains, where two examples separate the well-conditioned
case of distinct nodes from a confluent pair that represents a
derivative functional. Section~\ref{sec:square} takes up the square,
with its corner clustering and root-exponential convergence. Finally,
section~\ref{sec:potential} shows how the very same boundary
construction handles logarithmic and Cauchy volume potentials, for
targets both inside the domain and just outside it.

\section{The construction}
\label{sec:construction}

Let $z(t)$, $t \in [0, 2\pi)$, be a positively oriented
parametrization of $\gamma$, and write
$\omega(t)=W(z(t),\overline{z(t)})$ for the boundary data. We sample
$\gamma$ at $M$ values of $t$, chosen fine enough to resolve $\omega$; a
feature of scale $\delta$ near the boundary requires spacing below
$\delta$. All computations reported here use AAA with the
sign-blended weight choice of \cite{TrefethenWilber2025}. For a
prescribed maximal degree, AAA applied to the sampled data returns a set
of candidate poles. We discard the
nonfinite ones, and any pole within $10^{-8}d_M$ of a boundary sample,
where $d_M=\max_{j,k}|z_j-z_k|$ is the sampled boundary diameter. The
remaining poles are classified as interior or exterior with respect to
the sampled boundary, and we then solve the least-squares problem
\begin{equation}
\label{eq:ls}
\min_{c,\, d,\, p}\;
\sum_{j=1}^{M} \Bigl|
\sum_{z_k \in \Omega} \frac{c_k}{z_j - z_k}
+ \sum_{z_k \notin \Omega} \frac{d_k}{z_j - z_k}
+ \sum_{\ell=0}^{L} p_\ell\, z_j^\ell
- \omega_j \Bigr|^2 ,
\end{equation}
\rev{with $L=10$. On the starfish of section~\ref{sec:classical}, for each of
$L=4,10,16$ the best relative error is below $2\times10^{-15}$ over
AAA degrees $30,60,\ldots,150$. In all computations we use
$(z_j/s)^\ell$, $s=\max_j|z_j|$, scale each column to unit 2-norm,
and solve by truncated singular value decomposition, retaining the singular values strictly greater than} \rev{$\tau=\max\{M,N\}\,\texttt{eps}(\sigma_1)$, where $N$ is the number of
columns, $\sigma_1$ the largest singular value of the scaled matrix,
and $\texttt{eps}(\sigma_1)$ the floating-point spacing at $\sigma_1$.
This is MATLAB's default \texttt{rank} tolerance. The coefficients are
then returned in the original partial-fraction basis.} \rev{In the degree-16 fit on the
ellipse, the column of an exterior pole at distance $3\times10^{10}$
lies within $1.5\times10^{-15}$ of the polynomial span. Discarding one
singular value reduces the condition number from $4\times10^{15}$ to
$1.9\times10^4$, the value for the interior-pole block.}
The exterior poles and the polynomial improve the fit but integrate to
zero against analytic integrands, so only the interior coefficients
survive into $I_n=\pi\sum c_kf(z_k)$. A side benefit of fitting \rev{the coefficients} in this way is that \rev{the
pole--residue form is obtained directly: the least-squares fit supplies
the residues at the AAA poles, including when those poles are clustered}.

\begin{revblock}
The augmentation of a pole basis by a low-degree polynomial is standard
in lightning solvers \cite{GopalTrefethen2019}, and its effect on
least-squares fitting near singularities was analysed by Herremans,
Huybrechs and Trefethen \cite{HHT2023}. To measure the effect of the auxiliary
columns, Table~\ref{tab:ablation} fits the same AAA poles with four
bases.
\rev{For the ellipse, $S(z)=\kappa z+H(z)$ with $\kappa=(a-b)/(a+b)$
and $H(z)=O(1/z)$. All 23 poles retained at degree 24 are interior, so
the ellipse comparison isolates the effect of the polynomial, which
supplies the linear term $\kappa z$. On the starfish, adding exterior
poles to the interior-only basis reduces both the cubature error and
$\Lambda_n$, and the polynomial reduces the error further.}
\end{revblock}

\begin{table}[t]
\centering
\caption{\rev{Effect of the auxiliary columns in \eqref{eq:ls}: relative
error for $f=e^z$ and condition number $\Lambda_n$ when the same AAA
poles are fitted with four bases. Ellipse: degree 24, $n=23$ interior
poles; starfish: degree 150, $n=68$.}}
\label{tab:ablation}
\rev{\begin{tabular}{lcccc}
\toprule
 & \multicolumn{2}{c}{ellipse} & \multicolumn{2}{c}{starfish}\\
\cmidrule(lr){2-3}\cmidrule(lr){4-5}
basis & rel.\ err. & $\Lambda_n$ & rel.\ err. & $\Lambda_n$\\
\midrule
interior poles only              & $6.7\times10^{-2}$  & 1.067 & $1.4\times10^{-1}$  & 1.314\\
interior poles + polynomial      & $8.8\times10^{-16}$ & 1.000 & $1.7\times10^{-2}$  & 1.105\\
interior + exterior poles        & $6.7\times10^{-2}$  & 1.067 & $1.4\times10^{-12}$ & 1.002\\
full basis \eqref{eq:ls}         & $8.8\times10^{-16}$ & 1.000 & $6.5\times10^{-16}$ & 1.002\\
\bottomrule
\end{tabular}}
\end{table}

\rev{Alongside the fitting grid we use an offset validation grid
and write $\eps_{\rm val}$ for the maximum of $|\rho-W|$ on
that grid, where $\rho$ is the fitted rational function.} A convenient
global measure of the cancellation in the resulting rule is the quantity
\[
  \Lambda_n=\frac{\sum_k|\lambda_k|}
                   {\operatorname{area}(\Omega)},
  \qquad \lambda_k=\pi c_k.
\]
\rev{For $w\equiv1$ the two terms of this ratio are norms:
$\sum_k|\lambda_k|$ is the norm of the map
$(v_k)\mapsto\sum_k\lambda_kv_k$ on nodal values with the maximum norm,
and $\operatorname{area}(\Omega)$ is the norm of area integration on
$C(\bar\Omega)$. Thus $\Lambda_n$ is the absolute condition number of the
rule with respect to the nodal values, normalized by the norm of the
integral. An error of size $\delta$ in the function values changes $I_n$
by at most $\Lambda_n\operatorname{area}(\Omega)\,\delta$, as with a
Lebesgue constant in interpolation. For a general integrable weight $w$, integration against $w$ has
norm $\iint_\Omega |w|\dA$ on $C(\bar\Omega)$, and this replaces
$\mathrm{area}(\Omega)$ in the normalization of $\Lambda_n$. The
sensitivity of the rule to errors in the nodal values is
$\sum_k|\lambda_k|$, as before. For a rule that also uses derivative
values, the errors in those values enter the bound as well; see
section~\ref{sec:qd}.} For a positive rule
\rev{whose weights sum to the area} we have $\Lambda_n=1$, whereas large
values \rev{indicate cancellation and sensitivity to perturbations}.

In practice, when the exact integral is unknown, we raise the AAA
degree until $\eps_{\rm val}$ reaches the requested tolerance and stays
there as the validation grid is refined. For a given integrand,
$\tfrac12|\gamma|\eps_{\rm val}\max_\gamma|f|$ is then an a posteriori
error indicator, and it becomes a rigorous bound once the continuous
residual itself has been bounded. We refine the fitting grid whenever
the fitting and validation residuals begin to separate, or a
near-boundary feature is left unresolved. If $\Lambda_n$ grows rapidly, or a confluent cluster appears, we \rev{refit the rule in the confluent basis} of section~\ref{sec:qd}; the threshold at which this happens is problem
dependent.

\begin{revblock}
\paragraph{Admissible antiderivatives}
The identity \eqref{eq:pompeiu} is Green's theorem in complex form, and
some regularity of $W$ is needed for it to hold on a rectifiable Jordan
curve. We call $W$ an \emph{admissible} $\bar\partial$-antiderivative
of $w$ if either
\begin{enumerate}
\item[(i)] $W$ is continuously differentiable in a neighbourhood of
$\bar\Omega$ with $\partial W/\partial\bar z=w$ there; or
\item[(ii)] $W$ satisfies (i) except at finitely many points
$z_0\in\Omega$, near each of which $w$ is integrable and
$|W(z)|=o(|z-z_0|^{-1})$.
\end{enumerate}
Under (i), Green's theorem for continuously differentiable functions on
regions with rectifiable Jordan boundary \cite{Apostol1957}, applied to
$fW$, whose $\bar\partial$-derivative is $fw$ because $f$ is
analytic, gives \eqref{eq:pompeiu}. Under (ii), one excises a disk of
radius $\epsilon$ about each exceptional point, applies (i) to the remaining region, divided by cross-cuts into regions
with rectifiable Jordan boundary, the contributions of the cuts
cancelling, and lets $\epsilon\to0$: the circle contributes at
most $2\pi\epsilon\max|f|\max_{|z-z_0|=\epsilon}|W|\to0$, and the area
integral converges because $w$ is integrable. For $w\equiv1$ and
$W=\bar z$, condition (i) holds trivially; the antiderivatives of the
potential kernels in section~\ref{sec:potential} satisfy (i) for
targets outside $\bar\Omega$ and (ii) for targets inside, as verified
there.
\end{revblock}

\begin{theorem}\label{thm:cert}
Let $\Omega$ be a bounded domain whose boundary $\gamma$ is a
rectifiable Jordan curve, positively oriented, let $f$ be analytic in
a neighbourhood of $\bar\Omega$, and let $W$ be an \rev{admissible}
$\bar\partial$-antiderivative of $w$. Let
\[
\rho(z)
= \sum_{z_k \in \Omega} \frac{c_k}{z-z_k}
+ \sum_{z_k \notin \Omega} \frac{d_k}{z-z_k}
+ p(z)
\]
be a rational approximant to $W$ on $\gamma$, where all poles are
simple, all poles lie off $\gamma$, and $p$ is a polynomial. Set
$\eps = \max_{z \in \gamma} |\rho(z) - W(z,\bar z)|$. Then the rule
$I_n = \pi \sum_{z_k \in \Omega} c_k f(z_k)$ satisfies
\[
|I - I_n| \;\le\; \tfrac{1}{2}\, |\gamma| \, \eps \,
\max_{z \in \gamma} |f(z)| ,
\]
where $|\gamma|$ is the length of $\gamma$.
\end{theorem}

\begin{proof}
On a rectifiable Jordan curve, Cauchy's theorem applies to functions
analytic in $\Omega$ and continuous on $\bar\Omega$ \cite{Walsh1933},
and the residue formula used next follows on subtracting from $f\rho$
its principal parts at the poles in $\Omega$.
Since $f$ is analytic in a neighbourhood of $\bar\Omega$ and the poles of $\rho$ avoid
$\gamma$, residue calculus gives
$\frac{1}{2i}\oint_\gamma f \rho \, dz = \pi \sum_{z_k \in \Omega}
c_k f(z_k) = I_n$: each interior pole contributes
$2\pi i\, c_k f(z_k)/2i$, and the exterior poles and the polynomial
part contribute zero. Subtracting this from \eqref{eq:pompeiu}, \rev{which
holds because $W$ is admissible,}
\[
I - I_n = \frac{1}{2i} \oint_\gamma f(z) \,
\bigl(W(z,\bar z) - \rho(z)\bigr)\, dz .
\]
Therefore
\[
|I-I_n|
\leq \frac12 \int_\gamma |f(z)|\, |W(z,\bar z)-\rho(z)|\, |dz|
\leq \tfrac12 |\gamma|\, \eps \max_{z\in\gamma}|f(z)|,
\]
which proves the estimate.
\end{proof}

Theorem~\ref{thm:cert} calls for the continuous supremum of
the residual on $\gamma$, which a finite sample bounds only from below. A
fully rigorous certificate would require something more, such as an
adaptive or interval-based maximization along the boundary. The independent validation-grid checks return residuals consistent with those on the fitting grid. In fact the cubature error is
often much smaller than the indicator suggests\rev{: taking absolute values and applying the $L^1$--$L^\infty$ bound discards
cancellation in the contour integral. For $f=e^z$, the ratio of the actual error to
$\tfrac12|\gamma|\eps_{\rm val}\max_\gamma|f|$ lies between $10^{-8}$
and $10^{-4}$ on the starfish of section~\ref{sec:classical} at degrees
$30,60,\ldots,150$, and between $10^{-5}$ and $10^{-2}$ on the square
of section~\ref{sec:square} at degrees $20,40,\ldots,220$.}

\paragraph{A dual interpolatory interpretation}
For pure area integrals there is a second way to look at the rule,
through the exterior area Cauchy transform
\[
   H(\xi) = \frac{1}{\pi}\iint_\Omega \frac{\dA(z)}{\xi-z},
   \qquad \xi \in \C \setminus \bar\Omega .
\]
This function is analytic outside $\bar\Omega$ and behaves like
$H(\xi)=\operatorname{area}(\Omega)/(\pi\xi)+O(\xi^{-2})$ at infinity.
The part of the fitted rational function that matters here is
\[
   r_n(z) = \sum_{z_k\in\Omega} \frac{c_k}{z-z_k}.
\]
Taking the Cauchy-kernel integrand $f_\xi(z)=1/(z-\xi)$ with $\xi$
outside $\bar\Omega$, the definition of $H$ and a direct evaluation of
the rule give
\begin{equation}
\label{eq:dual-cauchy}
   I(f_\xi) = -\pi H(\xi), \qquad
   I_n(f_\xi) = -\pi r_n(\xi), \qquad
   I(f_\xi)-I_n(f_\xi)=\pi\{r_n(\xi)-H(\xi)\}.
\end{equation}
In other words, the interior cubature rule is at the same time a
rational approximation to the exterior area Cauchy transform. For the
ellipse, where $H$ is known in closed form\rev{, $H(\xi)=S(\xi)-\kappa\xi$
with $S$ the exterior branch of the Schwarz function,} one can check
this identity by hand\rev{, and numerically: at degree 16 and eighteen
exterior points $\xi$, the two sides of \eqref{eq:dual-cauchy} agree to
$7\times10^{-15}$ while the error itself ranges from
$2\times10^{-16}$ to $2\times10^{-9}$}.

The next proposition shows that, under a condition on the zeros of
$r_n-H$, the rule is interpolatory. It is the area analogue of the
result in \cite[Section~11]{HorningTrefethen2026}.

\begin{proposition}[Conditional rational exactness]
\label{prop:dual}
Let $z_1,\ldots,z_n$ be the distinct interior cubature nodes, where
$n$ is their number, and set
\[
  r_n(z)=\sum_{k=1}^n\frac{c_k}{z-z_k},
  \qquad D(z)=r_n(z)-H(z).
\]
Suppose that $D$ has distinct zeros $s_1,\ldots,s_n$ in
$\C\setminus\bar\Omega$. For any $f$ analytic in a neighbourhood of
$\bar\Omega$, let
\[
  q(z)=\sum_{j=1}^n\frac{\alpha_j}{z-s_j}
\]
be the unique rational function of this form satisfying
$q(z_k)=f(z_k)$. Then
\[
  I(q)=I_n(q)=I_n(f).
\]
\end{proposition}

\begin{proof}
\rev{The interpolant $q$ exists and is unique because the Cauchy matrix
$C_{kj}=1/(z_k-s_j)$ has determinant
$\prod_{k<j}(z_k-z_j)(s_j-s_k)\big/\prod_{k,j}(z_k-s_j)$, up to sign,
which is nonzero since the $z_k$ are distinct, the $s_j$ are distinct,
and the two sets are disjoint.}
Since the poles of $q$ lie outside $\bar\Omega$, Fubini's theorem and
the Cauchy formula give
\[
\frac{1}{2i}\oint_\gamma q(\xi)H(\xi)\,d\xi
 =\frac{1}{2\pi i}\iint_\Omega
   \oint_\gamma \frac{q(\xi)}{\xi-z}\,d\xi\,dA(z)
 =\iint_\Omega q(z)\,dA(z).
\]
Residue calculus also gives
$I_n(q)=(2i)^{-1}\oint_\gamma q(z)r_n(z)\,dz$. Hence
\[
I(q)-I_n(q)
  =-\frac{1}{2i}\oint_\gamma q(z)D(z)\,dz.
\]
The exterior poles of $q$ are cancelled by the zeros of $D$, so $qD$
is analytic in the exterior of $\gamma$. Moreover,
\[
  q(z)=O(z^{-1}),\qquad D(z)=O(z^{-1}),
\]
and therefore $q(z)D(z)=O(z^{-2})$. Deforming $\gamma$ to a large
circle in the exterior shows that the large-contour contribution
vanishes. Thus
$I(q)=I_n(q)$, while the interpolation conditions give
$I_n(q)=I_n(f)$.
\end{proof}

\rev{The proposition is conditional: whether $D$ has $n$ zeros outside
$\bar\Omega$ is a property of the computed fit. The condition can be checked
a posteriori by locating $n$ distinct zeros of $D$ there. When it holds,
the AAA cubature rule integrates exactly every linear combination of
the kernels $1/(z-s_j)$ at those $n$ exterior points, which is the area
analogue of the interpolatory characterization of one-dimensional
rules. The exterior-pole and
polynomial terms of \eqref{eq:ls} drop out of this statement: they
integrate to zero against analytic integrands, so
$I_n$ depends on $r_n$ alone, and that is why the proposition is stated
for $r_n$ rather than for the full fit $\rho$.} Expanding
\eqref{eq:dual-cauchy} at infinity, finally, gives
\[
   D(z)=\frac{\sum_k c_k-\operatorname{area}(\Omega)/\pi}{z}
        +O(z^{-2}),
\]
so that exact reproduction of the area improves $D$ to $O(z^{-2})$ and
allows broader normalizations at infinity; the strictly proper $q$ of
Proposition~\ref{prop:dual} already satisfies $qD=O(z^{-2})$ as it
stands. The whole conditional interpretation
carries over to other weights, with $H$ replaced by the exterior Cauchy
transform of $w\dA$.

 The identity \eqref{eq:pompeiu} needs $W$ to be single-valued on
$\bar\Omega$ and \rev{admissible in the sense above}; the
antiderivatives of section~\ref{sec:potential} are chosen to be
single-valued and bounded at the target, so that the identity holds as
written whether the target lies inside or outside $\Omega$. As for the
constant, with $\gamma$ positively oriented the factor $\pi$ in
\eqref{eq:rule} is just $2\pi i/(2i)$, and \rev{for $w\equiv1$ and $f \equiv 1$ the rule returns
$\sum_k\lambda_k$, which by Theorem~\ref{thm:cert} agrees with
$\mathrm{area}(\Omega)$ to within $\tfrac12|\gamma|\eps$, a
convenient check to run at execution time (on the square of
section~\ref{sec:square}, $|\sum_k\lambda_k-4|=2.9\times10^{-11}$ at
the finest degree).}
One must still watch for Froissart-type pole--zero pairs. A spurious
interior pole with a small least-squares coefficient is usually
harmless, though in an ill-conditioned basis its coefficient can be
large; a confluent
pair with large opposite weights, on the other hand, can be entirely
genuine and is best handled \rev{by the confluent refit} of
section~\ref{sec:qd}.

The weight enters the construction only through the values of $W$ on
the boundary; section~\ref{sec:potential}
makes use of two different $\bar\partial$-antiderivatives for potential
kernels.

\section{Classical examples}
\label{sec:classical}

Unless we say otherwise, relative error means $|I_n-I|/|I|$. When a
maximum is taken over several integrands, we switch to the scaled error
$|I_n-I|/\max(1,|I|)$, which keeps the comparison meaningful when a
reference value is zero or very small. Table~\ref{tab:summary} collects the
smooth, quadrature-domain, and polygonal examples of this and the next
two sections; the volume potentials of
section~\ref{sec:potential} have their own tables.

\begin{table}[t]
\centering
\caption{The smooth, quadrature-domain, and polygonal examples: the
number $n$ of nodes carrying the rule and \rev{the reported relative
error, or maximal scaled error where several integrands are used.
The confluent-refit row uses degree 16.}}
\label{tab:summary}
\begin{tabular}{lcc}
\toprule
example & $n$ & rel.\ err.\\
\midrule
disk, mean-value rule                     & 1   & $3.6\times10^{-15}$\\
ellipse, focal Gauss--Gegenbauer          & 38  & $8.0\times10^{-16}$\\
starfish, adaptive two-dimensional reference & 55  & $4.1\times10^{-16}$\\
quadrature domain, confluent pair         & 2   & $8.6\times10^{-11}$\\
\rev{quadrature domain, confluent refit}  & \rev{1 (double)} & \rev{$2.4\times10^{-15}$}\\
Neumann's oval, separated nodes           & 2   & $8.8\times10^{-15}$\\
square, closed form / tensor Gauss        & \rev{158} & $1.6\times10^{-11}$\\
\bottomrule
\end{tabular}
\end{table}

\paragraph{Disk} On a circle of radius $R$ centered at $z_c$ the
boundary data is already rational, since
$\bar z = \bar z_c + R^2/(z - z_c)$ on $\gamma$. The only non-negligible
interior contribution collapses onto the center, with weight $\pi R^2$,
while whatever additional poles the algorithm produces carry negligible
weight. Rational approximation of the boundary data thus reproduces the classical mean value property.

\begin{figure}[t]
\centering
\includegraphics[width=\textwidth]{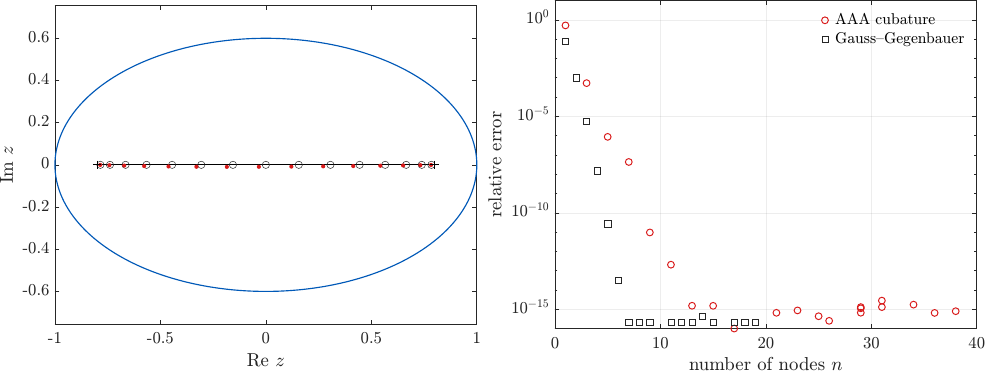}
\caption{Ellipse with semiaxes $(1, 0.6)$. Left: the AAA cubature
nodes (dots) cluster near the focal segment; open circles are the
Gauss--Gegenbauer points of the exact reduced rule; crosses mark the
foci. Right: relative error of the AAA cubature rule and of
Gauss--Gegenbauer on the focal segment, for $f = e^z$.}
\label{fig:ellipse}
\end{figure}

\paragraph{Ellipse} Take the ellipse with semiaxes $(a, b) = (1, 0.6)$
and $c^2 = a^2 - b^2$. Here the Schwarz function has a branch cut along
the focal segment $[-c, c]$, and shrinking the contour of
\eqref{eq:pompeiu} down onto that cut gives the exact reduction
\begin{equation}
\label{eq:focal}
\iint_E f \dA \;=\; \frac{2ab}{c^2} \int_{-c}^{c} f(x)\,
\sqrt{c^2 - x^2}\, dx ,
\end{equation}
a Gegenbauer (Chebyshev second-kind) weight on the focal segment. We
verified \eqref{eq:focal} independently. As Figure~\ref{fig:ellipse}
shows, the AAA nodes cluster near the focal segment, and the figure compares
their convergence with Gauss--Gegenbauer quadrature.

\begin{revblock}
 Both rules are
constructed from the domain alone, independently of $f$; what differs is the analyticity
each of them is able to exploit. The ellipse $\gamma$ is itself the
Bernstein ellipse $E_{R_0}$ of the focal segment with
$R_0=(a+b)/c$, equal to $2$ here. By \eqref{eq:focal}, the exterior
Cauchy transform $H$ is proportional to the Cauchy transform of the
Gegenbauer measure on $[-c,c]$. Let $x_j$ and $\omega_j$ be the nodes
and area-integral weights of the $n$-point Gauss--Gegenbauer rule in
\eqref{eq:focal}. Applying that rule to $x\mapsto1/(z-x)$ gives
\[
 H_n(z)=\frac{1}{\pi}\sum_{j=1}^n\frac{\omega_j}{z-x_j},
 \qquad \max_{z\in\gamma}|H(z)-H_n(z)|=O(R_0^{-2n})
\]
\cite{GautschiVarga1983}. Since $\bar z=\kappa z+H(z)$ on $\gamma$,
the rational function $\kappa z+H_n(z)$ supplies a boundary approximant
with this error. This gives an achievable benchmark for the best fit
with $n$ interior poles and a linear polynomial part.
For the computed AAA fits, the measured validation residual decays
like $e^{-1.36n}$ over degrees $6$ to $22$, consistent with the
benchmark $R_0^{-2n}=e^{-1.39n}$. This rate describes the experiment;
Theorem~\ref{thm:cert} converts a bound on each continuous boundary
residual into an upper bound for the cubature error, uniform over
analytic $f$ with $\max_\gamma|f|\le1$.

Gauss--Gegenbauer applied to an individual $f$ exploits its analyticity
relative to the focal segment. If $f$ is analytic in a neighbourhood of the closed Bernstein
ellipse $E_R$, with $R>1$, its quadrature error is $O(R^{-2n})$, with
a constant depending on $f$ and $R$; for the entire function $e^z$, any
fixed $R$ is admissible, giving superexponential convergence.
The AAA cubature error for a particular $f$ can also be smaller than
its boundary-residual bound through cancellation in the contour
integral. When the exact reduction \eqref{eq:focal} is available,
Gauss--Gegenbauer provides the more economical rule in
Figure~\ref{fig:ellipse}.

The AAA rule offers instead independence from a geometric model. Figure~\ref{fig:perturbed} shows an experiment in which the
boundary is perturbed to $z(t)=a\cos t+ib\sin t+p\,e^{3it}$, with
$p=10^{-2}$ and $p=10^{-4}$, and the reference value is the boundary
integral \eqref{eq:pompeiu} over the perturbed curve, computed by the
trapezoidal rule at two resolutions. Gauss--Gegenbauer on the focal
segment of the unperturbed ellipse, with the prefactor of
\eqref{eq:focal}, stalls at the level $O(p)$ of the geometry error,
whereas the AAA rule, working from the perturbed boundary data, drives
on down to machine precision. On a domain known to be perturbed one
would of course use a rule built for that domain; the comparison shows
only that the AAA rule adapts to the boundary as it actually is.
\end{revblock}

\begin{figure}[t]
\centering
\includegraphics[width=\textwidth]{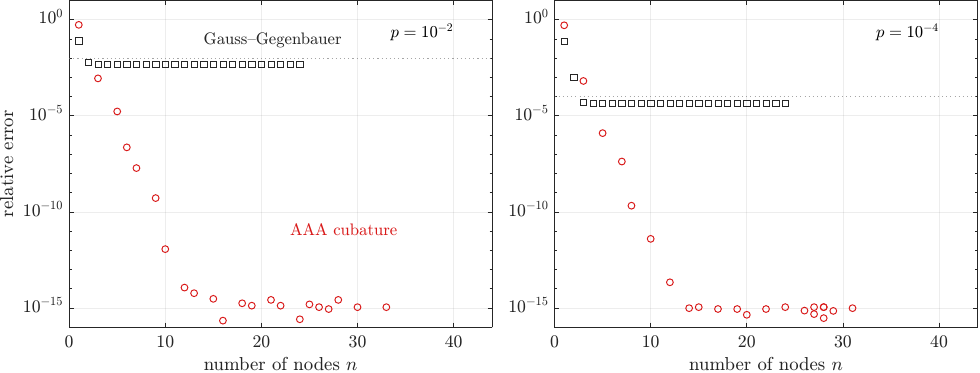}
\caption{The ellipse perturbed by $p\, e^{3it}$, with $p = 10^{-2}$
(left) and $p = 10^{-4}$ (right); function $f = e^z$.
Gauss--Gegenbauer on the focal segment of the underlying ellipse
floors at the geometry error $O(p)$ (dotted line at height $p$); the
AAA cubature rule, built from the true boundary data, converges to
machine precision.}
\label{fig:perturbed}
\end{figure}

\begin{figure}[t]
\centering
\includegraphics[width=\textwidth]{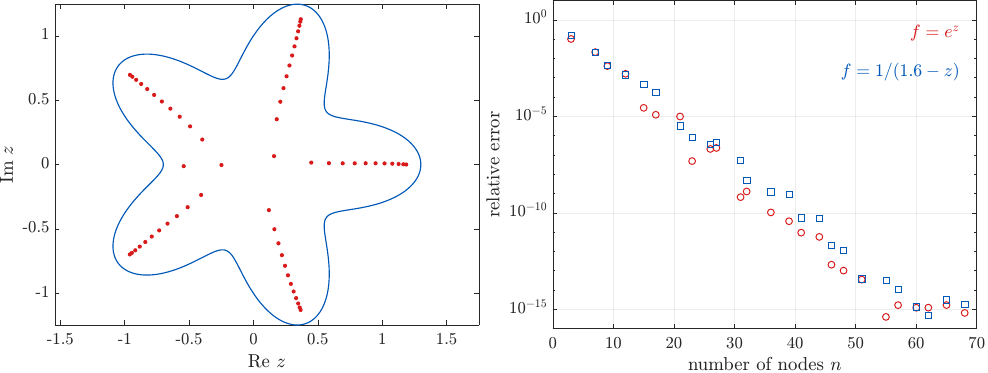}
\caption{Starfish domain $r(t) = 1 + 0.3\cos 5t$. Left: the cubature
nodes trace a five-armed mother-body-like analytic skeleton of the
domain. Right: relative error of the AAA cubature rule for an entire
integrand and for an integrand with a pole at $z = 1.6$.}
\label{fig:starfish}
\end{figure}

\paragraph{Starfish} For the domain $r(t) = 1 + 0.3 \cos 5t$ the
Schwarz function is considerably more complicated, with a more elaborate set of interior singularities. Figure~\ref{fig:starfish} shows the outcome: the nodes trace a five-armed star inside the
domain, a candidate mother-body-like skeleton of $\Omega$
\cite{Zidarov,Gustafsson1998}, and the rule converges geometrically all
the way to machine precision. As an independent check, the boundary
reference agrees with two-dimensional adaptive integration over
$\Omega$, carried out separately on the real and imaginary parts.

\section{Quadrature domains}
\label{sec:qd}

Recall that a domain $\Omega$ is a quadrature domain if there are
finitely many points $a_j \in \Omega$ and coefficients $c_{jk}$, only
finitely many of them nonzero, for which
\[
\iint_\Omega f \dA \;=\; \sum_j \sum_{k \ge 0} c_{jk} f^{(k)}(a_j)
\qquad \text{for all integrable analytic } f,
\]
or, equivalently, if its Schwarz function is meromorphic in $\Omega$
\cite{AharonovShapiro1976,Sakai1982,GustafssonShapiro2005}. The
simplest nontrivial example is $\Omega = \phi(\mathbb{D})$ with
$\phi(\zeta) = \zeta + a\zeta^2$, $0 < a < 1/2$, whose Schwarz
function has a single double pole at $0$ and therefore obeys the exact
identity
\begin{equation}
\label{eq:asidentity}
\iint_\Omega f \dA \;=\; \pi \bigl[ (1 + 2a^2)\, f(0) + a\, f'(0) \bigr].
\end{equation}
We confirmed \eqref{eq:asidentity} against an independent boundary
reference.

\begin{figure}[t]
\centering
\includegraphics[width=\textwidth]{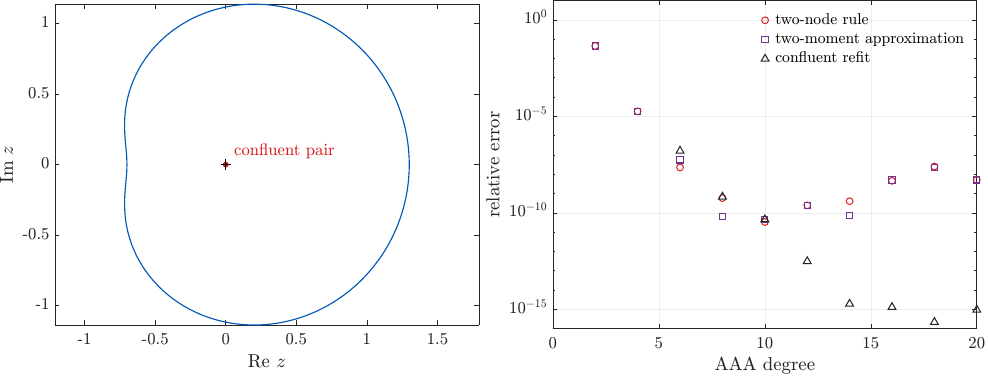}
\caption{The quadrature domain $\phi(\mathbb{D})$,
$\phi(\zeta) = \zeta + 0.3\zeta^2$. Left: the discovered rule consists
of two nodes $1.4 \times 10^{-5}$ apart with large opposite weights, a
confluent pair representing the node of order two at the origin in
\eqref{eq:asidentity}\rev{; the cross is the cluster centre $z_C$}.
Right: relative error for $f=e^z$ against AAA degree \rev{for the
two-node rule as computed, for its two-moment Taylor approximation,
and for the refit in the confluent basis
\eqref{eq:refit}}. \rev{Two interior nodes are present from degree $2$.
They first satisfy both confluent-pair criteria at degree $6$.}}
\label{fig:qd}
\end{figure}

\begin{revblock}
How does a rule with simple poles reproduce a derivative functional?
\rev{A \emph{confluent pair} of simple poles approximates the double pole of
the Schwarz function. At degree 10 the nodes are $1.4\times10^{-5}$ apart,
with nearly opposite weights of modulus $6.7\times10^4$ and
$\Lambda_n=3.6\times10^4$.} For a cluster $C$ of nodes with centre $z_C$, specified below, define
the \emph{local moments}
\[
   \mu_j = \sum_{k\in C} \lambda_k (z_k-z_C)^j .
\]
\rev{At degree 10,} $\mu_0$ and $\mu_1$ match the two
coefficients $\pi(1+2a^2)$ and $\pi a$ of \eqref{eq:asidentity} to
$1.4\times10^{-10}$ and $2.9\times10^{-10}$: the large opposite weights
are a finite-difference stand-in for the derivative functional. The
reason the moments are the right quantities is Taylor's theorem.
Let $B_C$ be the closed disk about $z_C$ of radius
$r_C=\max_{k\in C}|z_k-z_C|$, and suppose $f$ is analytic in a
neighbourhood of $B_C$. \rev{Taylor's formula with integral remainder along each segment from
$z_C$ to $z_k$ gives}
\[
   \sum_{k\in C} \lambda_k f(z_k)
   = \sum_{j=0}^{r-1} \frac{\mu_j}{j!} f^{(j)}(z_C)
     + E_r, \qquad
   |E_r|\le \frac{\max_{B_C}|f^{(r)}|}{r!}
       \sum_{k\in C}|\lambda_k|\,|z_k-z_C|^r ,
\]
\rev{where $r=2$ gives the Taylor truncation bound for the two-moment approximation.
For cluster detection we use}
\[
\eta_{\rm sep}(C)
 =\frac{\max_{j,k\in C}|z_j-z_k|}
        {\operatorname{diam}(\Omega)},
\qquad
\eta_{\rm cancel}(C)
 =\frac{|\sum_{k\in C}\lambda_k|}
        {\sum_{k\in C}|\lambda_k|},
\]
and, in the experiment of Figure~\ref{fig:qd}, choose the two nodes
with largest weight magnitudes as the candidate pair. We flag the pair
when $\eta_{\rm sep}<10^{-3}$ and $\eta_{\rm cancel}<10^{-2}$, and record
its share of $\Lambda_n$ alongside these diagnostics. The pair meets
both thresholds from degree 6 on (at degree 10, $\eta_{\rm sep}=6\times10^{-6}$
and $\eta_{\rm cancel}=3\times10^{-5}$, and the pair carries all of
$\Lambda_n$). A production code would need to test such thresholds more
widely; here they serve to guide interpretation. Once a cluster has
been spotted we set
\[
   z_C=\frac{\sum_{k\in C}|\lambda_k|z_k}
             {\sum_{k\in C}|\lambda_k|}
\]
\rev{and approximate the pair by $\mu_0f(z_C)+\mu_1f'(z_C)$.
Alternatively, we refit the two cluster columns of \eqref{eq:ls} in the basis}
\begin{equation}
\label{eq:refit}
   \frac{1}{z-z_C}\quad\text{and}\quad\frac{1}{(z-z_C)^2},
\end{equation}
\rev{keeping the other columns. The pair contributes $Q_C(f)=\pi[c_1f(z_C)+c_2f'(z_C)]$.}

\rev{For $f=e^z$ at degree 16, the relative errors in Figure~\ref{fig:qd}
are $4.6\times10^{-9}$ for the raw rule, $5.0\times10^{-9}$ for its
two-moment approximation, and $1.3\times10^{-15}$ for the confluent refit.
The two-moment approximation retains the coefficient errors of the original
fit. The largest scaled error of the refit over the five test integrands
is $2.4\times10^{-15}$.}

\rev{If errors in function values and derivatives are bounded by $\delta_0$
and $\delta_1$, respectively, then}
\[
 |\Delta Q_C|\le\pi|c_1|\delta_0+\pi|c_2|\delta_1.
\]
At degree 16 the two sensitivity coefficients are $3.7071$ and $0.94248$.
Additional simple nodes contribute $\sum_k|\lambda_k|\delta_0$.
A single normalized condition number for this mixed data requires a
chosen relative scaling of values and derivatives.

For validation, we evaluate the full refitted rational function on an
independent boundary grid. This evaluation includes the exterior poles
and the polynomial part. We set
\[
 \eps_{\rm val}=\max_j\left|\rho_{\rm refit}(z_j^{\rm val})
 -W(z_j^{\rm val},\overline{z_j^{\rm val}})\right|.
\]
In this example the validation points are
$z_j^{\rm val}=\phi(e^{2\pi i(j+1/2)/M_v})$, $0\le j<M_v$.
\rev{At degree 16 the validation residual is about $1.8\times10^{-14}$
on both offset grids ($M_v=3200,6400$). It provides an error indicator.
For cluster centres inside $\Omega$, the residue formula gives
$I(f)-Q_{\rm refit}(f)=(2i)^{-1}\oint_\gamma f(W-\rho_{\rm refit})\,dz$,
so the continuous-residual bound of Theorem~\ref{thm:cert} applies.}
Gustafsson, He, Milanfar and Putinar \cite{GHMP2000} reconstruct a
quadrature domain from its moments; the computation here goes the
other way, from the domain to the nodes and coefficients of its identity.

\end{revblock}

\begin{figure}[t]
\centering
\includegraphics[width=\textwidth]{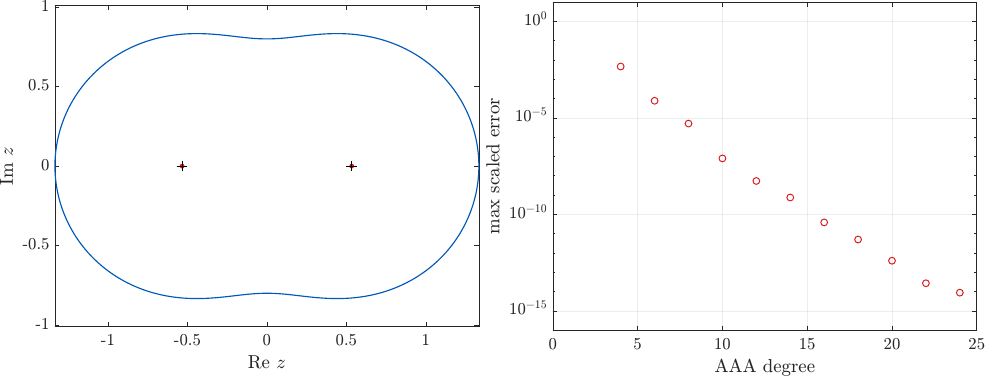}
\caption{Neumann's oval, $\phi(\zeta) = \zeta/(1 - 0.25\zeta^2)$.
Left: the two computed nodes (dots) and the exact nodes $\pm 8/15$ of
\eqref{eq:neumann} (crosses). Right: maximal scaled error over five
integrands against the AAA degree; the number of poles placed
inside $\Omega$ is exactly two at every degree.}
\label{fig:neumann}
\end{figure}

\rev{A quadrature domain with \emph{separated} simple nodes gives a rule
using function values at distinct points.} Neumann's oval is $\Omega = \phi(\mathbb{D})$ with
$\phi(\zeta) = \zeta/(1 - q^2 \zeta^2)$ and $0 < q < 1$
\cite{GustafssonShapiro2005}. This map is univalent on $\mathbb{D}$,
its Schwarz function has two simple poles, and residue calculus on the
unit circle yields the exact identity
\begin{equation}
\label{eq:neumann}
\iint_\Omega f \dA \;=\; \pi c \,\bigl[ f(x_*) + f(-x_*) \bigr],
\qquad
x_* = \frac{q}{1 - q^4}, \quad
c = \frac{1 + q^4}{2\,(1 - q^4)^2} ,
\end{equation}
with both nodes $\pm x_*$ lying interior to $\Omega$.
Figure~\ref{fig:neumann} treats the case $q=1/2$, for which $x_*=8/15$
and $c=136/225$. The two nodes come out correctly at every degree
shown, and raising the degree sharpens the exterior approximation and drives the integration error to machine precision.

\section{A domain with corners}
\label{sec:square}

\leavevmode\rev{Theorem~\ref{thm:cert}} holds for any rectifiable Jordan curve, and the
experiments in this paper all use smooth or piecewise-analytic
boundaries, with parametrizations and sampling grids fine enough to
resolve the boundary data.

At a corner the boundary data is no longer analytic. Along a straight edge through $a$ in the direction
$e^{i\theta}$ we have $\bar z=\bar a+e^{-2i\theta}(z-a)$, but two
adjacent edges continue to different functions\rev{, so the boundary data
is continuous but has a kink at each corner, of the kind of $|x|$ at
the origin. For the model singularity $|x|$, Newman's theorem establishes
root-exponential rational approximation, $O(e^{-C\sqrt n})$
\cite{Newman1964}. Exponentially clustered poles also underpin lightning
approximation of corner singularities \cite{GopalTrefethen2019,HHT2023}}. For the square the
observed nodes line up along the two diagonals, just as a mother-body
configuration would suggest \cite{Gustafsson1998}.

Figure~\ref{fig:square} shows the computation for the square
$[-1,1]^2$. \rev{The boundary sampling is graded toward the corners, with
240 points per half-edge spaced logarithmically between distance
$10^{-12}$ and $1$ from the corner, and the validation grid is offset by
half a step on the same logarithmic scale. Such grading is standard in
lightning solvers \cite{GopalTrefethen2019}: the fitted rational
function varies on the scale of the distance to the nearest clustered
pole, and a grid that is uniform in arclength would leave the fit
unconstrained between the innermost poles.} The
nodes cluster at the corners and hug the diagonals, while the exterior
poles run off along the corresponding corner rays. The closed-form
reference for $e^z$ and an independent tensor Gauss--Legendre reference
agree to machine precision. As the figure shows, the cubature error
falls root-exponentially\rev{, consistent with the rates known for
rational approximation of corner singularities; a proof of this rate
for the cubature error remains open, and the fitted exponent is an
observation.}

\begin{figure}[t]
\centering
\includegraphics[width=\textwidth]{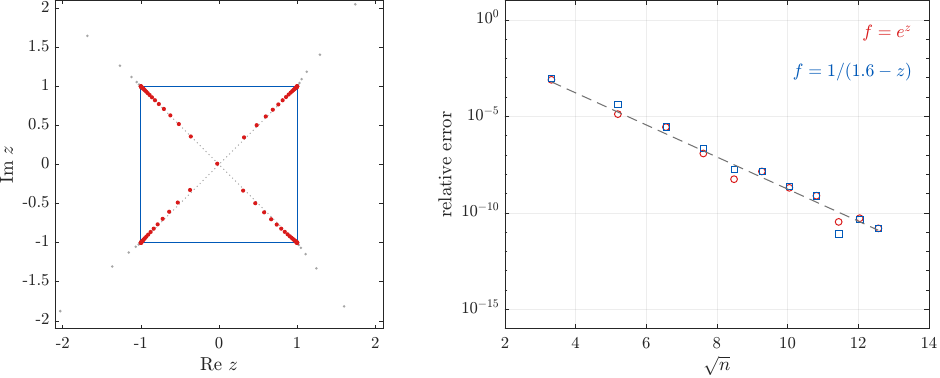}
\caption{The square $[-1,1]^2$. Left: cubature nodes (dots) align with
the two diagonal candidate mother-body segments, with exponential
clustering toward the corners; the gray dots are exterior poles, which
also line up on the corner rays of the exterior skeleton. Right: relative
error against $\sqrt{n}$ for two integrands; the dashed line is
the least-squares fit $\propto e^{\rev{-1.91}\sqrt{n}}$, the
root-exponential regime of lightning approximation.}
\label{fig:square}
\end{figure}

\section{Nearly singular volume potentials}
\label{sec:potential}

As a last application, consider the two-dimensional Newtonian potential
of an analytic integrand $f$ over $\Omega$,
\begin{equation}
\label{eq:newton}
\Phi(z_0) \;=\; \iint_\Omega f(z)\, \log|z - z_0| \dA ,
\end{equation}
which is singular when $z_0\in\Omega$ and nearly singular when $z_0$
sits close to $\gamma$. For general densities one commonly resorts to
volume meshes with local corrections, or to density interpolation
\cite{AndersonEtAl2024,ShenSerkh2024}; for analytic $f$, \rev{the
present construction avoids both}. The
$\bar\partial$-antiderivative
\begin{equation}
\label{eq:Wlog}
W(z, \bar z) \;=\; (\bar z - \bar z_0)\bigl( \log|z - z_0| -
\tfrac{1}{2} \bigr)
\end{equation}
is single-valued and continuous for $z_0$ inside or outside $\Omega$,
provided only that $z_0\notin\gamma$, and it treats interior and
exterior targets on the same footing. Near $\gamma$ the sampling must
of course resolve a boundary feature on the scale of the target
distance. \rev{For $z_0\notin\bar\Omega$ the function \eqref{eq:Wlog} is
smooth in a neighbourhood of $\bar\Omega$, so condition (i) of
section~\ref{sec:construction} holds. For $z_0\in\Omega$ it is smooth
except at $z_0$, where $w=\log|z-z_0|$ is integrable and
$|W|=O(\epsilon|\log\epsilon|)$ on the circle $|z-z_0|=\epsilon$, so
condition (ii) holds: excising a disk of radius $\epsilon$ about $z_0$,
the circle contributes $O(\epsilon^2|\log\epsilon|)$, which disappears
as $\epsilon\to0$. The same argument takes care of the bounded Cauchy
antiderivative introduced below, whose circle contribution is
$O(\epsilon)$.}

\begin{figure}[t]
\centering
\includegraphics[width=\textwidth]{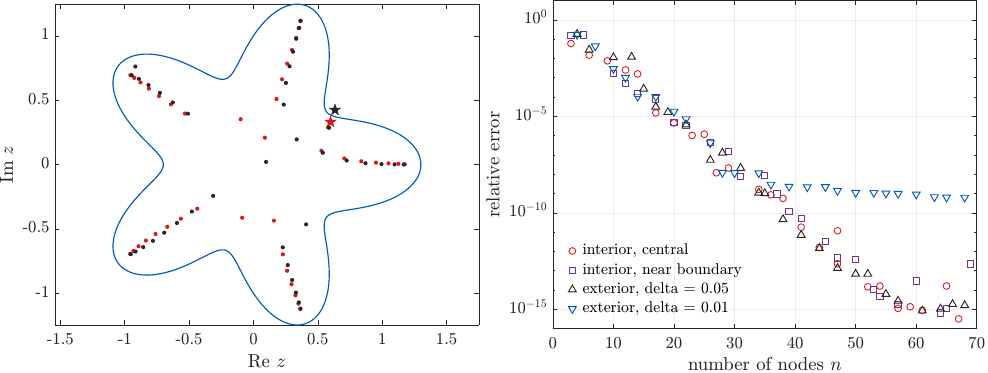}
\caption{Newtonian potential over the starfish body. Left: nodes for
an interior target near the boundary (red star: the nodes cluster at
the target) and for an exterior target at distance $0.05$ (black star:
the nodes cluster at a candidate reflected point inside the body); both
node sets also trace the mother-body-like analytic skeleton. Right: relative error against
the number of nodes for four targets; the exterior case at distance
$0.01$ is limited here by the $M = 800$ boundary sampling (see
Table~\ref{tab:resM}).}

\label{fig:potential}
\end{figure}

\begin{table}[t]
\centering
\caption{Volume potential \eqref{eq:newton} over the starfish body:
best relative errors for four targets, with $M = 800$ boundary
samples.}
\label{tab:res}
\begin{tabular}{lcc}
\toprule
target & rel.\ err. & $n$\\
\midrule
interior, central        & \rev{$3.2\times 10^{-16}$} & 67\\
interior, near boundary  & $6.5\times 10^{-16}$ & 64\\
exterior, $\delta=0.05$  & $8.5\times 10^{-16}$ & 61\\
exterior, $\delta=0.01$  & $6.2\times 10^{-10}$ & 68\\
\bottomrule
\end{tabular}
\end{table}

\begin{table}[t]
\centering
\caption{The exterior target at distance $\delta = 0.01$ as the
boundary sampling is refined: the accuracy floor is set by the
sampling.}
\label{tab:resM}
\begin{tabular}{ccc}
\toprule
$M$ & rel.\ err. & $n$\\
\midrule
800  & $6.2\times 10^{-10}$ & 68\\
1600 & $3.0\times 10^{-14}$ & 68\\
3200 & $2.2\times 10^{-16}$ & 63\\
\bottomrule
\end{tabular}
\end{table}

The results are shown in Figure~\ref{fig:potential} and
Tables~\ref{tab:res} and~\ref{tab:resM}. Every well-resolved target reaches machine
precision, and for the closest exterior target, refining the boundary
grid removes what at first looks like an accuracy floor.

The location of the nodes again reflects the continuation of the
boundary data. For an interior target
the boundary data \eqref{eq:Wlog} continues inward with a branch point
at $z_0$, and the nodes cluster at the target. For an exterior
target they cluster instead near a point consistent with a solution of
$S(z)=\bar z_0$ on the reflected branch relevant to this experiment;
for a circle this point is the classical inverse point of Kelvin.
The left panel of Figure~\ref{fig:potential} thus gives numerical evidence
of an image-point structure.

\rev{A different kernel enters the construction only through the
boundary data $W$.}
Replacing \eqref{eq:Wlog} by
$W = (\bar z - \bar z_0)/(z - z_0)$, which is bounded at $z_0$, produces
the area Cauchy integral $\iint_\Omega f(z)/(z - z_0) \dA$. \rev{The
same procedure applied to the new boundary data gives new nodes and
weights, and
the computed values again agree with the reference to near machine
precision.}

This same area Cauchy integral delivers a derivative of the potential.
For $z_0$ outside $\bar\Omega$, ordinary differentiation under the
integral sign applies. For $z_0\in\Omega$, one may split off a disk
centered at $z_0$ and pass to the limit, since $(z-z_0)^{-1}$ is locally
integrable. As
\[
  \frac{\partial}{\partial z_0}\log|z-z_0|
  =-\frac{1}{2(z-z_0)},
\]
we obtain
\[
\iint_\Omega \frac{f(z)}{z - z_0} \dA \;=\;
-2\, \frac{\partial \Phi}{\partial z_0} ,
\]
where $\partial/\partial z_0 = \tfrac12(\partial/
\partial x_0-i\,\partial/\partial y_0)$. For a real-valued potential
this reproduces the usual gradient components; for complex $f$ it is to
be read componentwise.

\section{Discussion}
\label{sec:discussion}

We have followed a single chain of reasoning from the Cauchy--Green
identity to boundary rational approximation, to interior poles and
residues, and finally to cubature. \rev{Theorem~\ref{thm:cert}} covers
analytic integrands and a fixed admissible $\bar\partial$-antiderivative;
general continuous integrands lie outside its scope. \rev{Conjugation
extends the construction to harmonic integrands for real weights.} The weights may be complex
or signed, the sampled residuals stay mere indicators until the
continuous residual has been bounded, and confluent clusters have to be
\rev{refitted in a confluent basis or approximated through their moments
with the Taylor remainder controlled}.

The main open question concerns the geometry of the nodes. Our
starfish and potential experiments give numerical evidence of a
candidate mother-body-like skeleton and of reflected image points, and it remains to prove that the computed sets are the corresponding
objects of potential theory. On the square the measured rate is consistent with
the root-exponential behavior of lightning approximation
\cite{GopalTrefethen2019}, but this rate has not been proved. For the derivative identities, a backward-error analysis of the
\rev{cluster detection and refit} remains open. \rev{And although
Table~\ref{tab:ablation} shows accurate rules with modest $\Lambda_n$ for
the full basis of \eqref{eq:ls}, a stability analysis of
the least-squares step that holds uniformly in the degree is likewise
open.} Each of these is an invitation for further work.

Several extensions suggest themselves. The boundary identity carries
over directly to domains with several boundary components, and the
pole classification for such domains is left for future work. General
polygons and curved corners will demand the same care with graded
sampling as the square did. Other kernels will call for their own
$\bar\partial$-antiderivatives, and families of targets for a
set-valued AAA approximation with shared poles. \rev{A three-dimensional
version is a different problem altogether: the construction rests on
residue calculus, a planar tool, and a three-dimensional construction
would need a mechanism of its own in its place.}

All computations were done in MATLAB with Chebfun's \texttt{aaa}. The
scripts are available at \url{https://github.com/zavala92/cubature_aaa}.

\section*{Acknowledgments}
We are grateful to Nick Trefethen for helpful comments on an earlier
draft, particularly concerning the analyticity assumptions and the role of
boundary approximation, and to Andrew Horning for discussions of the
dual interpretation and numerical stability. \rev{We also thank the two
anonymous referees for their comments and suggestions.}

\end{document}